\documentclass[smallextended]{svjour3}       
\smartqed  
\usepackage{mathptmx}      
\usepackage{amssymb,amsmath,newlfont,enumerate}

\numberwithin{theorem}{section}

\numberwithin{equation}{section}

\newcommand{\CC}{{\mathbb C}}
\newcommand{\DD}{{\mathbb D}}

\newcommand{\TT}{{\mathbb T}}

\newcommand{\cD}{{\mathcal D}}
\DeclareMathOperator{\hol}{\mathrm Hol}

\journalname{Journal}

\begin{document}

\title{Polynomial approximation in weighted Dirichlet spaces
\thanks{JM supported by an NSERC grant. TR supported by grants from NSERC and the Canada Research Chairs program.}}

\titlerunning{Polynomial approximation}        

\author{Javad Mashreghi        \and
        Thomas Ransford}

\institute{J. Mashreghi\at
              D\'epartement de math\'ematiques et de statistique, Universit\'e Laval,
Qu\'ebec City (Qu\'ebec), Canada G1V 0A6\\
              \email{javad.mashreghi@mat.ulaval.ca}
                  \and
           T. Ransford \at
              D\'epartement de math\'ematiques et de statistique, Universit\'e Laval,
Qu\'ebec City (Qu\'ebec),  Canada G1V 0A6\\
 \email{thomas.ransford@mat.ulaval.ca}  
}

\date{Received: date / Accepted: date}

\maketitle

\begin{abstract}
We give an elementary proof of an analogue of Fej\'er's theorem
in weighted Dirichlet spaces with superharmonic weights. This provides a simple way of seeing that
polynomials are dense in such spaces.
\end{abstract}

\section{Introduction and statement of main result}\label{S:intro}

Let  $\DD$ be the open unit disk and $\TT$ be the unit circle.
We denote $\hol(\DD)$ the set of all holomorphic functions on $\DD$,
and  by $H^2$ the Hardy space on $\DD$.

Given $\zeta\in\overline{\DD}$, we define $\cD_\zeta$ to be the set
 of all $f\in\hol(\DD)$ of the form 
\[
f(z)=a+(z-\zeta)g(z),
\]
where $g\in H^2$ and $a\in\CC$. In this case, we set $\cD_\zeta(f):=\|g\|_{H^2}^2$. We adopt the convention that,
if $f\in\hol(\DD)$ but $f\notin\cD_\zeta$, then $\cD_\zeta(f):=\infty$.

Given a positive finite Borel measure $\mu$ on $\overline{\DD}$, we define $\cD_\mu$ to be the set of all
$f\in\hol(\DD)$ such that
\[
\cD_\mu(f):=\int_{\overline{\DD}}\cD_\zeta(f)\,d\mu(\zeta)<\infty.
\]
We endow $\cD_\mu$ with the norm $\|\cdot\|_{\cD_\mu}$ defined by
\[
\|f\|_{\cD_\mu}^2:=|f(0)|^2+\cD_\mu(f).
\]

The spaces $\cD_\mu$  are in fact precisely the weighted Dirichlet spaces on $\DD$ with superharmonic weights. 
Though this identification is not required for a technical understanding of the results in this note,
 it serves as background and motivation for these results.
 The final section \S\ref{S:background} contains a brief account of these spaces and their history. 
When $\mu=\delta_\zeta$, the Dirac measure at $\zeta\in\overline{\DD}$, the space $\cD_\mu$ reduces to $\cD_\zeta$, which is sometimes called the local Dirichlet space at $\zeta$. 

A fundamental property of the spaces $\cD_\mu$ is that polynomials are dense.
This fact was originally established by Richter \cite{Ri91}
and  Aleman \cite{Al93}
using a  type of wandering-subspace theorem. Other proofs followed, most recently
a direct proof via a Fej\'er-type approximation theorem \cite{MR19}.
Our goal in this note is to give a more elementary proof of this last theorem.
In fact we shall establish the following generalization.

\begin{theorem}\label{T:main}
Let $(w_{n,k})_{n,k\ge0}$ be an array of complex numbers such that:
\begin{align}
w_{n,k}&=0 \quad(k>n),\label{E:main0}\\
\lim_{n\to\infty} w_{n,k}&=1 \quad(k\ge0), \label{E:main1}\\
|w_{n,k}|&\le M \quad(n,k\ge0), \label{E:main2}\\
|w_{n,k}-w_{n,k+1}|&\le L/n \quad(n,k\ge0), \label{E:main3}
\end{align}
where $L,M$ are constants.
Given $f\in\hol(\DD)$, say  $f(z)=\sum_{k=0}^\infty a_kz^k$, set
\begin{equation}\label{E:pn}
p_n(z):=\sum_{k=0}^n w_{n,k}a_kz^k.
\end{equation}
Then, for each positive finite measure $\mu$ on $\overline{\DD}$ such that $f\in\cD_\mu$, we have
\[
\|f-p_n\|_{\cD_\mu}\to0 \quad(n\to\infty).
\]
\end{theorem}

As a consequence, we deduce the Fej\'er-type theorem mentioned above.
Given $f(z)=\sum_{k=0}^\infty a_kz^k$, we write
\[
s_n(f)(z):=\sum_{k=0}^n a_kz^k
\quad\text{and}\quad 
\sigma_n(f)(z):=\sum_{k=0}^n \Bigl(1-\frac{k}{n+1}\Bigr)a_kz^k.
\]

\begin{corollary}
If $\mu$ is a finite measure on $\overline{\DD}$ and if $f\in\cD_\mu$, then 
\[
\|\sigma_n(f)-f\|_{\cD_\mu}\to0 \quad(n\to\infty).
\]
\end{corollary}

\begin{proof}
Apply Theorem~\ref{T:main} with $w_{n,k}:=1-k/(n+1)$ if $k\le n$ and zero otherwise.\qed
\end{proof}

By contrast, it is known that, if $\mu=\delta_1$,  the Dirac measure at $\zeta=1$, then there exists $f\in\cD_\mu$ such that 
$\|s_n(f)-f\|_{\cD_\mu}\not\to0$ as $n\to\infty$ (see for example \cite[p.117, Exercise~7.3.2]{EKMR14}). 
This shows that Theorem~\ref{T:main} is no longer true if we omit 
the condition~\eqref{E:main3}.


\section{Approximation in local Dirichlet spaces}\label{S:localapprox}

We begin with a simple lemma about approximation in $H^2$.

\begin{lemma}\label{L:idea1}
Let $(w_{n,k})_{n,k\ge0}$ be an array of complex numbers 
satisfying the conditions \eqref{E:main0}, \eqref{E:main1} and \eqref{E:main2}.
Let $g\in H^2$, say $g(z)=\sum_{k=0}^\infty b_kz^k$, and 
for $n\ge0$ let
\begin{equation}\label{E:gn}
g_n(z):=\sum_{k=0}^{n-1} w_{n,k+1}b_kz^k.
\end{equation}
Then
\begin{equation}\label{E:idea1a}
\|g-g_n\|_{H^2}\to0 \quad(n\to\infty).
\end{equation}
Moreover, there exists a constant $C$, depending only on the array $(w_{n,k})$, such that
\begin{equation}\label{E:idea1b}
\|g_n\|_{H^2}\le C\|g\|_{H^2}\quad(n\ge0).
\end{equation}
\end{lemma}

\begin{proof}
From condition~\eqref{E:main2} we have 
\[
\|g_n\|_{H^2}^2=\sum_{k=0}^\infty |w_{n,k+1}|^2|b_k|^2\le M^2\sum_{k=0}^\infty|b_k|^2=M^2\|g\|_{H^2}^2.
\]
This gives \eqref{E:idea1b} with $C=M$. 

Also, given $\epsilon>0$, we can choose $N$ large enough so that $\sum_{k=N}^\infty|b_k|^2<\epsilon^2$,
and then, for $n\ge N$, we have
\[
\|g-g_n\|_{H^2}^2\le \sum_{k=0}^{N-1}|1-w_{n,k+1}|^2|b_k|^2+(1+M)^2\epsilon^2.
\]
From \eqref{E:main1}, the first term on the right-hand side tends to zero as $n\to\infty$. 
Thus we have $\limsup_{n\to\infty}\|g-g_n\|_{H^2}\le (1+M)\epsilon$.
As $\epsilon$ is arbitrary, this gives \eqref{E:idea1a}.\qed
\end{proof}

In light of the definition of $\cD_\zeta$, 
this lemma translates into the following approximation result for local Dirichlet spaces.

\begin{theorem}\label{T:localapprox}
Let $(w_{n,k})_{n,k\ge0}$ be an array of complex numbers 
satisfying conditions \eqref{E:main0}, \eqref{E:main1} and \eqref{E:main2}.
Let $\zeta\in\overline{\DD}$ and let $f\in\cD_\zeta$,
say $f(z)=a+(z-\zeta)g(z)$, where $g\in H^2$ and $a\in\CC$.
Define $g_n$ as in \eqref{E:gn} and set
\begin{equation}\label{E:fn}
f_n(z):=a+(z-\zeta)g_n(z).
\end{equation}
Then 
\[
\cD_\zeta(f-f_n)\to 0 \quad(n\to\infty),
\]
and there exists a constant $C$, depending only on the array $(w_{n,k})$, such that
\[
\cD_\zeta(f_n)\le C^2\cD_\zeta(f) \quad(n\ge0).
\]
\end{theorem}

\begin{proof}
This is an immediate consequence of the identities 
 $\cD_\zeta(f_n)=\|g_n\|_{H^2}^2$ and $\cD_\zeta(f-f_n)=\|g-g_n\|_{H^2}^2$.\qed
\end{proof}

Let us compute the polynomials $f_n$ explicitly. If we write $f(z)=\sum_{k=0}^\infty a_kz^k$
and $g(z)=\sum_{k=0}^\infty b_kz^k$ and equate coefficients of $z^k$
in the relation $f(z)=a+(z-\zeta)g(z)$, then we obtain
\begin{equation}\label{E:ab}
\left\{
\begin{aligned}
a_0&=a-\zeta b_0,\\
a_k&=b_{k-1}-\zeta b_k \quad(k\ge1).
\end{aligned}
\right.
\end{equation}
Hence
\begin{align*}
f_n(z)
&=a+(z-\zeta)g_n(z)\\
&=a+(z-\zeta)\sum_{k=0}^{n-1}w_{n,k+1}b_k z^k\\
&=a+\sum_{k=1}^n w_{n,k}b_{k-1}z^k-\zeta\sum_{k=0}^{n}w_{n,k+1}b_kz^k\\
&=a+\sum_{k=1}^n w_{n,k}(b_{k-1}-\zeta b_k)z^k+\zeta\sum_{k=0}^{n}(w_{n,k}-w_{n,k+1})b_kz^k-\zeta w_{n,0}b_0\\
&=a+\sum_{k=1}^nw_{n,k}a_kz^k+\zeta\sum_{k=0}^{n}(w_{n,k}-w_{n,k+1})b_kz^k-\zeta w_{n,0}b_0,
\end{align*}
the last line using \eqref{E:ab}.
Rearranging this slightly, we get
\begin{equation}\label{E:fn2}
f_n(z)=\sum_{k=0}^nw_{n,k}a_kz^k+\zeta\sum_{k=0}^{n}(w_{n,k}-w_{n,k+1})b_kz^k+(1-w_{n,0})a.
\end{equation}

The following special case is worthy of note.

\begin{corollary}\label{C:localapprox}
Let $\zeta\in\overline{\DD}$ and let $f\in\cD_\zeta$, say
$f(z)=\sum_{k=0}^\infty a_kz^k$. Then $\sum_{k=0}^\infty a_k\zeta^k$ converges and, setting
\[
f_n(z):=\sum_{k=0}^{n-1} a_kz^k +\Bigl(\sum_{k=n}^\infty a_k\zeta^{k-n}\Bigr)z^n,
\]
we have
\[
\cD_\zeta(f-f_n)\to0.
\]
\end{corollary}

\begin{proof}
Let $w_{n,k}:=1$ for $k\le n$ and zero otherwise.
This satisfies \eqref{E:main0}, \eqref{E:main1} and \eqref{E:main2}, so
by Corollary~\ref{C:localapprox} we have $\cD_\zeta(f-f_n)\to0$ as $n\to\infty$. Formula \eqref{E:fn2} shows that
\[
f_n(z)=\sum_{k=0}^n a_kz^k+\zeta b_nz^n=\sum_{k=0}^{n-1}a_kz^k +b_{n-1}z^n,
\]
where the last equality is from \eqref{E:ab}. 

All that remains is to identify $b_{n-1}$. 
For this, we note that, by \eqref{E:ab}, for all $N\ge n$ we have
\[
\sum_{k=n}^N a_k\zeta^{k-n}
=\zeta^{1-n}\sum_{k=n}^N (b_{k-1}\zeta^{k-1}-b_k\zeta^k)
=b_{n-1}-b_N\zeta^{N+1-n}.
\]
Since the sequence $(b_k)$ is square summable, 
we also have $b_N\to0$ as $N\to\infty$. Hence $b_{n-1}=\sum_{k=n}^\infty a_k\zeta^{k-n}$, as desired.\qed
\end{proof}


\section{Proof of Theorem~\ref{T:main}}\label{S:proof}

Unfortunately, the polynomials $f_n$ that approximate $f$ 
in Theorem~\ref{T:localapprox} and Corollary~\ref{C:localapprox} depend upon
$\zeta$ (note that the $(b_k)$ also depend upon $\zeta$). 
This makes them unsuitable for approximation in $\cD_\mu$. To circumvent this difficulty, we return to the formula \eqref{E:fn2}. Notice that, although $f_n$ itself depends on $\zeta$,
the first term on the right-hand side of \eqref{E:fn2} does not. If we can somehow show that the other terms tend to zero  in an appropriate way, then the first term approximates $f_n$ and hence also~$f$.
This is the strategy for the proof of Theorem~\ref{T:main}. To implement it, we need the following general estimate.

\begin{lemma}\label{L:poly}
Let $q$ be  a polynomial of degree $n$. Then
\[
\cD_\zeta(q)\le n^2\|q\|_{H^2}^2 \quad(\zeta\in\overline{\DD}).
\]
\end{lemma}

\begin{proof}
We have $\cD_\zeta(q)=\sum_{k=0}^{n-1}|d_k|^2$, where the $(d_k)$ are determined by the relation
\[
q(z)=q(\zeta)+(z-\zeta)\sum_{k=0}^{n-1}d_kz^k.
\]
Writing $q(z)=\sum_{k=0}^nc_kz^k$, and 
equating coefficients of powers of $z$ gives
$c_n=d_{n-1}$ and $c_k=d_{k-1}-\zeta d_k$ for $1\le k\le n-1$. 
Solving for $d_k$ in terms of $c_k$, we get
\[
d_k=\sum_{j=k+1}^n c_j\zeta^{n-j} \quad(0\le k\le n-1).
\]
By the Cauchy--Schwarz inequality, it follows that
\[
|d_k|^2\le (n-k)\sum_{j=k+1}^n|c_j|^2\le n\|q\|_{H^2}^2
\quad(0\le k\le n-1).
\]
Hence, finally
\[
\cD_\zeta(q)=\sum_{k=0}^{n-1}|d_k|^2\le n^2\|q\|_{H^2}^2. 
\]
This completes the proof of the lemma.\qed
\end{proof}

\textit{Completion of proof of Theorem~\ref{T:main}.}
Let $\mu$ be a finite positive Borel measure on $\overline{\DD}$, and let $f\in\cD_\mu$. Then
\[
\int_{\overline{\DD}}\cD_\zeta(f)\,d\mu(\zeta)<\infty,
\]
and in particular $\cD_\zeta(f)<\infty$ for $\mu$-almost every $\zeta$ in $\overline{\DD}$. We claim that,  for each such $\zeta$, and with $p_n$ as defined as in \eqref{E:pn}, we have
\begin{equation}\label{E:claim}
\lim_{n\to\infty}\cD_\zeta(f-p_n)=0
\quad\text{and}\quad
\sup_{n\ge1}\cD_\zeta(f-p_n)\le C^2\cD_\zeta(f),
\end{equation}
where $C$ is a constant depending only on the array $(w_{n,k})$.
If so, then, by the dominated convergence theorem, we have
\[
\cD_\mu(f-p_n)=\int_{\overline{\DD}}\cD_\zeta(f-p_n)\,d\mu(\zeta)\to0
\quad(n\to\infty).
\]
Since also $f(0)-p_n(0)=a_0(1-w_{n,0})\to0$ as $n\to\infty$, it follows that
\[
\|f-p_n\|_{\cD_\mu}\to0
\quad(n\to\infty),
\]
thereby establishing the theorem.

It remains to verify the claim \eqref{E:claim}.
Fix $\zeta$ with $\cD_\zeta(f)<\infty$, and define 
$g_n$ and $f_n$ as in \eqref{E:gn} and \eqref{E:fn} respectively.
As $\cD_\zeta(\cdot)^{1/2}$ is a seminorm, we have
\[
\cD_\zeta(f-p_n)^{1/2}\le \cD_\zeta(f-f_n)^{1/2}+\cD_\zeta(f_n-p_n)^{1/2}.
\]
By Corollary~\ref{C:localapprox}, we have
\[
\lim_{n\to\infty}\cD_\zeta(f-f_n)=0
\quad\text{and}\quad
\sup_{n\ge1}\cD_\zeta(f-f_n)\le C_1^2\cD_\zeta(f),
\]
where $C_1$ depends only on $(w_{n,k})$.
Therefore it remains to show that
\begin{equation}\label{E:rts}
\lim_{n\to\infty}\cD_\zeta(f_n-p_n)=0
\quad\text{and}\quad
\sup_{n\ge1}\cD_\zeta(f_n-p_n)\le C_2^2\cD_\zeta(f),
\end{equation}
where $C_2$ depends only on $(w_{n,k})$.
For this we use the formula \eqref{E:fn2}, according to which
\[
f_n(z)-p_n(z)=\zeta\sum_{k=0}^{n}(w_{n,k}-w_{n,k+1})b_kz^k+\text{constant}.
\]
Using Lemma~\ref{L:poly} and the fact that $\cD_\zeta(\cdot)$ is zero on constants, we get
\[
\cD_\zeta(f_n-p_n)\le n^2 \sum_{k=0}^n|w_{n,k}-w_{n,k+1}|^2|b_k|^2.
\]
Using condition~\eqref{E:main3}, we have 
\[
\cD_\zeta(f_n-p_n)\le L^2\sum_{k=0}^n|b_k|^2\le L^2\|g\|_{H^2}^2=L^2\cD_\zeta(f),
\]
which yields the second part of \eqref{E:rts} with $C_2=L$.
As for the first part, given $\epsilon>0$, 
we can choose $N$ so large that $\sum_{k=N}^\infty|b_k|^2<\epsilon^2$. 
Then, for all $n\ge N$, we have
\begin{align*}
\cD_\zeta&(f_n-p_n)^{1/2}\\
&\le \cD_\zeta\Bigl(\sum_{k=0}^{N-1}(w_{n,k}-w_{n,k+1})b_kz^k\Bigr)^{1/2}
+\cD_\zeta\Bigl(\sum_{k=N}^{n}(w_{n,k}-w_{n,k+1})b_kz^k\Bigr)^{1/2}\\
&\le \Bigl(N^2 \sum_{k=0}^{N-1}|w_{n,k}-w_{n,k+1}|^2|b_k|^2\Bigr)^{1/2}
+  \Bigl(n^2 \sum_{k=N}^n|w_{n,k}-w_{n,k+1}|^2|b_k|^2\Bigr)^{1/2}\\
&\le \Bigl(N^2 \sum_{k=0}^N|w_{n,k}-w_{n,k+1}|^2|b_k|^2\Bigr)^{1/2}
+ L\epsilon.
\end{align*}
From condition~\eqref{E:main1}, the first term on the right-hand side tends to zero as $n\to\infty$. 
Thus $\limsup_{n\to\infty}\cD_\zeta(f_n-p_n)\le L\epsilon$.
As $\epsilon$ is arbitrary, this establishes the first part of \eqref{E:rts},
and completes the proof of the theorem.\qed


\section{Background on weighted Dirichlet spaces}\label{S:background}

The purpose of this section is to explain the origin of the spaces $\cD_\mu$ and their connection to weighted Dirichlet spaces. As mentioned in the introduction, though this is not required for a technical understanding of the results in this note, it serves as background and motivation for these results.

Given a positive integrable function $\omega$  on $\DD$,
the $\omega$-weighted Dirichlet space consists of those $f\in\hol(\DD)$ such that
\[
\int_\DD|f'(z)|^2\,\omega(z)\,dA(z)<\infty,
\]
where $dA$ denotes normalized area measure on $\DD$.
The classical Dirichlet space corresponds to taking $\omega\equiv1$.

One class of weights that has been much studied over the years are the power weights $\omega(z):=(1-|z|^2)^{1-\alpha}$, where $0\le\alpha\le 1$.
One can show that, if $f(z)=\sum_{k=0}^\infty a_kz^k$, then
\[
\int_\DD |f'(z)|^2 (1-|z|^2)^{1-\alpha}\,dA(z)
\asymp \sum_{k=1}^\infty k^{\alpha}|a_k|^2.
\]  
Thus, as $\alpha$ runs from $0$ to $1$, we obtain a scale of spaces between the Hardy space ($\alpha=0$) and the classical Dirichlet space ($\alpha=1$).

Another important class of weights $\omega$ are the harmonic weights.
These were introduced by Richter
\cite{Ri91} in connection with his analysis of shift-invariant subspaces
of the classical Dirichlet space, and further studied by Richter and Sundberg in \cite{RS91}.

Subsequently  Aleman  \cite{Al93} introduced the class of superharmonic weights, 
which subsumes  the power weights and the harmonic weights.

Let $\omega$ be a positive superharmonic function on $\DD$.
Then there exists a unique positive finite Borel measure $\mu$ 
on $\overline{\DD}$ such that, for all $z\in\DD$,
\[
\omega(z)=\int_\DD \log\Bigl|\frac{1-\overline{\zeta}z}{\zeta-z}\Bigr|\frac{2}{1-|\zeta|^2}\,d\mu(\zeta)
+\int_\TT\frac{1-|z|^2}{|\zeta-z|^2}\,d\mu(\zeta)
\]
(see e.g.\ \cite[Theorem~4.5.1]{Ra95}).
Defining  $\cD_\mu(f)$  as in \S\ref{S:intro},
we have the following result,
which was established by Richter and Sundberg \cite[Proposition~2.2]{RS91} in the case when $\omega$ is harmonic on $\DD$ (this corresponds to  $\mu$ being supported on~$\TT$) and by Aleman \cite[\S IV, Theorem 1.9]{Al93} for general superharmonic weights.

\begin{theorem}\label{T:w=mu}
Let $\omega$ be a superharmonic weight on $\DD$,
and let $\mu$ be the associated measure on $\overline{\DD}$.
Then, for all $f\in\hol(\DD)$, we have
\[
\cD_\mu(f)=\int_\DD |f'(z)|^2\omega(z)\,dA(z).
\]
\end{theorem}

Thus $\cD_\mu$ is exactly the $\omega$-weighted Dirichlet space,
justifying the assertion made in the introduction.

Further information on weighted Dirichlet spaces can be found in \cite{RS91} and   \cite{Al93}, 
as well as in Chapter~7 of the monograph \cite{EKMR14}.

\section*{Conflict of interest}
The authors declare that they have no conflict of interest.

\bibliographystyle{spmpsci}      
\bibliography{biblist}

\begin{thebibliography}{1}
\providecommand{\url}[1]{{#1}}
\providecommand{\urlprefix}{URL }
\expandafter\ifx\csname urlstyle\endcsname\relax
  \providecommand{\doi}[1]{DOI~\discretionary{}{}{}#1}\else
  \providecommand{\doi}{DOI~\discretionary{}{}{}\begingroup
  \urlstyle{rm}\Url}\fi

\bibitem{Al93}
Aleman, A.: The multiplication operator on {H}ilbert spaces of analytic
  functions (1993).
\newblock Habilitationsschrift, Fern Universit\"at, Hagen

\bibitem{EKMR14}
El-Fallah, O., Kellay, K., Mashreghi, J., Ransford, T.: A {P}rimer on the
  {D}irichlet {S}pace, \emph{Cambridge Tracts in Mathematics}, vol. 203.
\newblock Cambridge University Press, Cambridge (2014)

\bibitem{MR19}
Mashreghi, J., Ransford, T.: {H}adamard multipliers on weighted {D}irichlet
  spaces (2019).
\newblock Preprint

\bibitem{Ra95}
Ransford, T.: Potential {T}heory in the {C}omplex {P}lane, \emph{London
  Mathematical Society Student Texts}, vol.~28.
\newblock Cambridge University Press, Cambridge (1995)

\bibitem{Ri91}
Richter, S.: A representation theorem for cyclic analytic two-isometries.
\newblock Trans. Amer. Math. Soc. \textbf{328}(1), 325--349 (1991)

\bibitem{RS91}
Richter, S., Sundberg, C.: A formula for the local {D}irichlet integral.
\newblock Michigan Math. J. \textbf{38}(3), 355--379 (1991)

\end{thebibliography}

\end{document}